\documentclass[11pt]{amsart}
\usepackage[margin=1in]{geometry}
\usepackage{amsmath,amssymb,amsthm,booktabs,microtype}
\usepackage{tikz}
\usepackage[most]{tcolorbox}
\usepackage[colorlinks=true,linkcolor=blue,urlcolor=blue,citecolor=blue,
hypertexnames=false]{hyperref}

\newtheorem{theorem}{Theorem}
\newtheorem{corollary}{Corollary}
\newtheorem{lemma}[theorem]{Lemma}

\newcommand{\Sph}{\mathbb S^2}
\newcommand{\cP}{\mathcal P}

\begin{document}

\title{A set of points with optimal $L^2$ spherical cap discrepancy}

\author{Carlos Beltr\'an, Jordi Marzo and Joaquim Ortega-Cerd\`a}

\address{ Departamento de Matem\'aticas, Estad\'{\i}stica y Computaci\'on, 
Universidad de Can\-ta\-bria, Avd. Los Castros s/n, 39005, Santander, Spain}
\email{beltranc@unican.es}

\address{Departament de Matem\`atiques i Inform\`atica, Universitat de 
Barcelona \& Barcelona Graduate School of Mathematics, Gran Via 585, 08007, 
Barcelona, Spain}
\email{jmarzo@ub.edu}

\address{Departament de Matem\`atiques i Inform\`atica, Universitat de 
Barcelona \& Barcelona Graduate School of Mathematics, Gran Via 585, 08007, 
Barcelona, Spain}
\email{jortega@ub.edu}

\thanks{The last two authors are supported by grant PID2024-160033NB-I00 funded by MICIU/AEI/10.13039/501100011033 and by FEDER, UE. JOC is also supported by the 2024 ICREA 00142 grant by the Generalitat de Catalunya.}

\begin{abstract}

We introduce an explicit deterministic collection of $N$ spherical points, $\cP_N\subset\mathbb S^2$, that we call the {\em deterministic Diamond points}. For every $0<\alpha<2$ we prove that there exists a constant $C_\alpha>0$ such that 
\[ 0\leq \frac{2^{\alpha+1}}{\alpha+2}N^2-\sum_{x,y\in\cP_N}|x-y|^\alpha \le C_\alpha N^{1-\alpha/2}. \]
This result shows that $\cP_N\subset\mathbb S^2$ has a Riesz energy deficit of optimal order in this range of the parameter. In particular, for $\alpha=1$, Stolarsky's invariance principle implies that $\cP_N\subset\mathbb S^2$ has $L^2$ spherical cap discrepancy of asymptotically optimal order
\[ D_{L^2}^C(\mathcal P_N)\asymp N^{-3/4}, \] 
making it, to our knowledge, the first explicit configuration proved to have this property. More generally, our result implies that the point set $\cP_N$  has Sobolev $H^{s}(\Sph)$ worst-case error of optimal order for all $1<s<2$.
\end{abstract}
\maketitle

\section{Introduction and results}

Many different quantities have been used to measure the regularity of discrete point sets. One of the most popular is the so-called $L^2$ spherical cap discrepancy, defined for a finite set $X\subset \Sph$ by 
\[ D_{L^2}^C(X)=\sqrt{\int_{-1}^{1}\int_{\Sph}\left| \frac{\#(X\cap C(x,t))}{\#X}-\sigma(C(x,t)) \right|^2\, d\sigma(x)\,dt}, \]
where $C(x,t)=\{ y\in \Sph:\ x\cdot y\ge t \}$ is the spherical cap with center $x\in \Sph$ and geodesic radius $\arccos t$, and $\sigma$ denotes the normalized surface area measure on $\Sph$; see \cite{Beck, DT97, BHS19} and the references therein.

Of  particular interest is the asymptotic behavior of the discrepancy when the number of points tends to infinity. It was proved by Stolarsky \cite{Stolarsky}, using previous work by Ralph Alexander \cite{Alexander1972} , that there exists (non-explicit) 
$N$-points sequences $X_N\subset \Sph$ such that
\begin{equation}\label{upper_bound}
D_{L^2}^C(X_N)\le C N^{-3/4}.    
\end{equation}

The main ingredient in Stolarsky's proof is the subsequently called Stolarsky's invariance principle which gives a direct connection between the $L^{2}$ spherical cap discrepancy and the mutual sum of distances for points $x_1,\dots, x_N\in \Sph$ \cite{Stolarsky}:
$$4 D_{L^2}^C (\{ x_1,\dots, x_N \})^2=\int_{\Sph}\int_{\Sph}|x-y|d\sigma(x)d\sigma(y)-
\frac{1}{N^2}\sum_{i=1}^N \sum_{j=1}^N |x_i-x_j|,$$
see \cite{BDM18} for a recent elementary proof.

Improving earlier work by Stolarsky \cite{Stolarsky} and Harman \cite{Har82}, 
the upper bound (\ref{upper_bound}) was actually proved to be sharp by Beck in \cite{Beck} by showing that 
there exists a 
constant $c>0$ such that for all configurations of points
$$c N^{-3/4}\le D_{L^2}^C(X_N).$$
For a recent and elementary proof of this result, see \cite{BB25}.

Several authors have searched for configurations of points with asymptotic discrepancy of optimal order $N^{-3/4}$. For example, the expected discrepancy of jittered sampling and other random processes was proved to be of that order, see \cite{Beck,AZ15,BGM24}. 
It was also proved in \cite{BSSW14,Skr19} that well separated spherical designs of optimal order, whose existence was established in a nonconstructive way in \cite{BRV13,BRV15}, also have this desired property.

However, to date, for all explicit, nonrandom configurations, the known discrepancy bounds are still quite far from optimal, being at best of order $N^{-1/2}$, which is the same order expected for uniform iid chosen points,
see \cite{ABD12,Eta21,FHM23} for theoretical bounds and \cite{HMS} for numerical experiments.

In this paper, we introduce a simple construction that derandomizes the Diamond ensemble introduced in \cite{Diamond} and whose $L^\infty$ discrepancy was studied in \cite{Eta21}. We refer to the resulting configurations as {\em deterministic Diamond points} and prove that their $L^2$ spherical cap discrepancy has optimal asymptotic growth. Our approach is similar to the construction that we introduced together with Uju\'e Etayo in \cite{BEMOC21} and to those of \cite{BL22,LG26}. 
In \cite{BEMOC21}
it was proved that a deterministic and explicit family of points provides an example of the previously elusive set of zeros 
of well-conditioned polynomials defined 
by Shub and Smale, whose existence 
had earlier been established by probabilistic methods \cite{SS93}. The underlying idea of this paper is similar.

A detailed construction of the deterministic Diamond points and some of their properties will be given in the next section. Briefly, we choose 
a particular collection of parallels, symmetric with respect to the equator, and associate with each 
parallel a band centered on it. The number of points assigned to each band 
is chosen to be proportional to the area of the band and the points are placed equally spaced in the central parallel of each band.

To estimate the discrepancy of the deterministic Diamond points, we compare the discrete, semi-discrete (in parallels), and continuous contributions to 
the energy 
associated with the distance kernel. Our technique is valid for the more general problem of estimating the sum of the mutual distances to some power $0<\alpha<2$, which we will see also yields a corollary for Sobolev worst-case errors.


\begin{theorem}\label{thm:main}
Let $\cP_N\subset\mathbb S^2$ be the $N$-set of deterministic Diamond points.  For every $0<\alpha<2$ there exists a constant $C_\alpha>0$ such that
\begin{equation}\label{theo:ineq}
0\leq  I_\alpha N^2-\sum_{x,y\in\cP_N}|x-y|^\alpha
   \le C_\alpha N^{1-\alpha/2},
\end{equation}
where $I_\alpha=\iint_{\Sph\times\Sph}|x-y|^\alpha
 \,d\sigma(x)d\sigma(y)=\frac{2^{\alpha+1}}{\alpha+2}.$
\end{theorem}

The bound in \eqref{theo:ineq} is of optimal order since the classical estimates of Wagner \cite{Wag90,Wag92} give
\begin{equation}\label{eq:optimalorder} 
I_\alpha N^2-\max_{\#X=N} \sum_{x,y\in X}|x-y|^\alpha \asymp N^{1-\alpha/2}.
\end{equation} 
For a simpler proof of the lower bound in the left-hand side of (\ref{eq:optimalorder}), and for a study of the conjectured constant accompanying the term $N^{1-\alpha/2}$ in \eqref{eq:optimalorder}, see \cite{BHS12,BHS19}. We also mention the recent work \cite{Steinerberger2026} for an upper bound attaining the conjectured quadratic coefficient for the exponent $\alpha=-2$, which is not in our range, but whose methods might also lead to sequences with proven low discrepancy points.

Taking $\alpha=1$ in equation \eqref{theo:ineq}, from Stolarsky's invariance principle we immediately get that the deterministic Diamond points have spherical $L^2$ discrepancy of optimal order:

\begin{corollary}\label{cor:main}
Let $\cP_N\subset\mathbb S^2$ be the $N$-set of deterministic Diamond points. There exist constants
$C,c>0$ such that
$$c N^{-3/4}\le D_{L^2}^C(\mathcal P_N)\le C N^{-3/4}.$$
\end{corollary}

Specializing \cite{BSSW14} to $\mathbb S^2$, we now introduce the notion of worst-case error for Sobolev spaces and the relation with our previous results. Given an integer $\ell\ge 0$, let $\mathcal{H}_\ell$ be the $(2\ell+1)$--dimensional vector space of eigenfunctions of the Laplace-Beltrami operator in $\mathbb S^2$, $\Delta,$ with eigenvalue $\lambda_\ell=\ell (\ell+1)$,
\begin{equation*}
-\Delta Y=\lambda_\ell Y,\;\;\;\; Y\in \mathcal{H}_{\ell}.
\end{equation*}
The space $\bigoplus_{\ell=0}^L \mathcal{H}_\ell$ of spherical harmonics of degree at most 
$L$ on $\mathbb{S}^2$ has dimension $(L+1)^2$ and coincides with the space of degree at most $L$ polynomials on $\mathbb R^3$ restricted to $\mathbb S^2$.

Let $L^2(\mathbb{S}^2)$ be the Hilbert space of square integrable real functions in $\mathbb{S}^2$
with the inner product
\[
\langle f,g \rangle=\int_{\mathbb{S}^2} f(x)g(x)\,d\sigma(x),\;\;\; f,g\in L^2(\mathbb{S}^2).
\]
Then $L^2(\mathbb{S}^2)=\bigoplus_{\ell\ge 0} \mathcal{H}_\ell$ and the Fourier series expansion of a function $f \in L^2(\mathbb{S}^2)$ is given by
\begin{equation*}
f=\sum_{\ell,k} f_{\ell,k} Y_{\ell,k},\qquad f_{\ell,k}=\langle f,Y_{\ell,k} \rangle=\int_{\mathbb{S}^2} f\, Y_{\ell,k} \,d\sigma,
\end{equation*}
where $\{ Y_{\ell,k} \}_{k=-\ell}^{\ell}$ is an orthonormal basis of $\mathcal{H}_\ell.$ 

For $s\ge 0,$ the $L^2(\mathbb{S}^2)$-based Sobolev spaces of order $s$ are the Hilbert spaces
$$\mathbb{H}^{s}(\mathbb{S}^2)=\left\{ f\in L^2(\mathbb{S}^2)\;\;:\;\; 
\sum_{\ell=0}^{+\infty}\sum_{k=-\ell}^{\ell} (1+\lambda_\ell)^s |f_{\ell,k}|^2<+\infty \right\},$$
with the norm
\begin{equation}\label{eq:sobolev}\| f \|_{\mathbb{H}^{s}(\mathbb{S}^2)}=\left( \sum_{\ell=0}^{+\infty}\sum_{k=-\ell}^{\ell} (1+\lambda_\ell)^s |f_{\ell,k}|^2  \right)^{1/2}.\end{equation}
The space 
$\mathbb{H}^{s}(\mathbb{S}^2)$ is continuously embedded in the space of continuous functions $\mathcal{C}(\mathbb{S}^2)$ if $s >1.$ Our last result shows that deterministic Diamond points have optimal worst-case error for Sobolev spaces $H^s(\Sph)$ if $1<s<2.$ As in the case of the $L^2$ cap discrepancy, this is, to our knowledge, the first explicit configuration to 
achieve this property.

\begin{corollary}\label{cor:wce}
Let $\cP_N\subset\mathbb S^2$ be the $N$-set of deterministic Diamond points. For $1<s<2$ there exists a constant
$C_s>0$ such that
$$\sup_{\| f \|_{\mathbb{H}^{s}(\mathbb{S}^2)}\le 1}  
\left| \frac{1}{N}\sum_{x\in \cP_N}f(x)-\int_{\mathbb{S}^2}f(x)d\sigma(x) \right| \le C_s N^{-s/2}.$$
\end{corollary}

The asymptotic behavior in the corollary above is optimal because 
there is a lower bound for the worst-case error of the same order, see  \cite{BSSW14}. Corollary \ref{cor:wce} follows immediately from Theorem \ref{thm:main} and the following relation proved in \cite[(42)]{BSSW14} for $1<s<2$ (note that \cite{BSSW14} uses several norms in $\mathbb{H}^{s}(\mathbb{S}^2)$, but they are all equivalent to that of \eqref{eq:sobolev}): for some constant $C_s>0$,
$$\sup_{\| f \|_{\mathbb{H}^{s}(\mathbb{S}^2)}\le 1}  
\left| \frac{1}{N}\sum_{x\in X_N}f(x)-\int_{\mathbb{S}^2}f(x)d\sigma(x) \right|^2\leq C_s\left(
I_{2s-2}-\frac{1}{N^2}\sum_{i=1}^N\sum_{j=1}^N|x_i-x_j|^{2s-2}\right).$$

\section{Notation and construction of the deterministic Diamond points}
We will use the asymptotic notation: for two collections of numbers $A(\eta,\nu,\ldots)$ and $B(\eta,\nu,\ldots)$, where $\eta,\nu,\ldots$ are any parameters, the expression
\[
A(\eta,\nu,\ldots)\asymp B(\eta,\nu,\ldots)
\]
means that there exist universal constants $0<c<C$ such that
\[
c A(\eta,\nu,\ldots)\leq B(\eta,\nu,\ldots)\leq CA(\eta,\nu,\ldots),\quad \text{for all choices of the parameters.}
\]
If we only have  the upper bound $B(\eta,\nu,\ldots)\leq CA(\eta,\nu,\ldots)$ then we write
\[
B(\eta,\nu,\ldots)\lesssim A(\eta,\nu,\ldots)
\]
The symbols $\asymp_\eta$ and $\lesssim_\eta$ are used if the constants depend on the parameter $\eta$.

The  unspecified constants in Theorem \ref{thm:main} allow us to assume in the proof that $N\geq N_0$ for any fixed positive integer $N_0$. We now give the construction of our point set for every $N\geq4$, although the proof will assume $N\geq 1024$, that implies $M\geq16$ in \eqref{eq:rj} below. 
The construction is similar to that of \cite[Section 4.1]{Diamond} or \cite[Section 3]{LG26} but excludes the poles and avoids the choice of random phases. Fix
$N\ge4$ and put
\begin{equation}\label{eq:rj}
 M=\left\lfloor\sqrt{N/4}\right\rfloor,\qquad
 r_j=4j\quad(1\le j<M),\qquad
 r_M=N-2\sum_{j=1}^{M-1}r_j .
\end{equation}

Note that
\[
4M^2\leq N\leq 4(M+1)^2-1=4M^2+8M+3,\text{ and then }N\asymp M^{2}.
\]
Extend symmetrically by $r_{M+j}=r_{M-j}$ for $1\le j<M$.

We denote the parallels in $\Sph$ by 
$$Q_{h}=\{(x,y,z)\in \Sph : z=h \},\;\;\;-1\leq h \leq 1.$$ 
Let
\[
H_j=1-\frac{2}{N}\sum_{k=1}^j r_k\quad 0\leq j\leq 2M-1,
\]
which define the bands
$$\mathcal B_j=\{ (x,y,z)\in \Sph \;\:\; H_j\le z\le H_{j-1} \},$$
where $\mathcal B_1,\mathcal B_{2M-1}$ are spherical caps. Then $\Sph=\bigcup_{j=1}^{2M-1} \mathcal B_j$ and 
\[
\sigma(\mathcal B_j)=\frac{H_{j-1}-H_j}{2}=\frac{r_j}{N},\quad 1\leq j\leq 2M-1.
\] 
The parallels bounding each band $\mathcal B_j$ are precisely $Q_{H_{j-1}}$ and $Q_{H_{j}}$. We consider also the central parallels $Q_{h_j}$ with heights
\[
h_j=\frac{H_{j-1}+H_j}{2}=H_{j-1}-\frac{r_j}{N}=H_{j}+\frac{r_j}{N}=1-\frac{2}{N}\sum_{k=1}^{j-1}
r_k 
-\frac{r_j}{N},
\]
for $1\leq j\leq 2M-1,$ and observe that $h_M=0$ and $h_{M+j}=-h_{M-j}$ for $j=1,\dots, M-1.$

 In order to construct $\cP_N$, place $r_j$ equally spaced points in each parallel of height $h_j$. Hence, these points are the $r_j$ vertices of a regular polygon, equivalently, up to rotation and scaling, they form a set of roots of unity in the circumference defined by the
parallel. The points can be rotated in any desired way, or forced to have one of them in any predefined meridian. In order to fix some criteria, we determine the set completely by forcing that in each parallel there exists one point such that its first coordinate is positive and its second coordinate vanishes, but we point that this choice is not important for the proof below, and hence our main result and corollaries also hold for any arbitrary rotation of the points in each parallel.

For later use, observe that
\begin{align*}
h_j=&1-\frac{4j^2}N,\quad 1\leq j\leq M-1,\\
 \sum_{j=1}^{M-1}r_j
 =&4\frac{(M-1)M}{2}=2M(M-1),\\
 r_M&=N-4M(M-1),
\end{align*}
and hence $4M\le r_M\le12M+3.$ Thus the number of points in the equator is of order $M$, and the total number of points is $N$. There are a total of $2M-1$ occupied parallels. See Figure \ref{fig:bemoc} for a graphical representation of the  points.

Each parallel $Q_j=Q_{h_j}$ is a circumference of radius $\rho_j=\sqrt{1-h_j^2}$. Note that


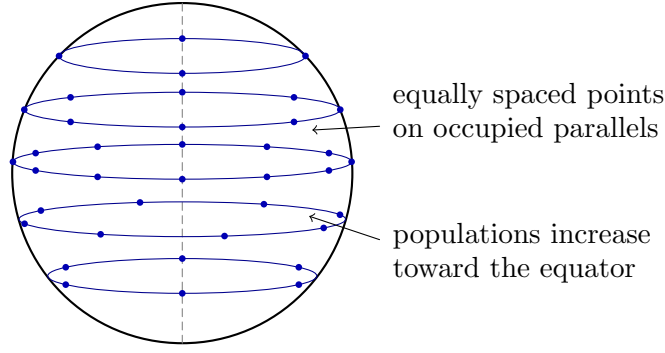
\begin{figure}[t]
\centering
\begin{tikzpicture}[scale=1.0,line cap=round,line join=round]
  \def\R{2.25}
  \draw[thick] (0,0) circle (\R);
  \draw[densely dashed,gray] (0,-\R) -- (0,\R);
  \foreach \y/\rx/\step/\shift in
   {1.55/1.63/90/0,.84/2.09/45/0,.15/2.24/30/0,
    -.61/2.16/45/15,-1.36/1.78/60/30}{
    \draw[blue!55!black] (0,\y) ellipse [x radius=\rx,y radius=.23];
    \foreach \a in {0,\step,...,359}
      \fill[blue!70!black]
       ({\rx*cos(\a+\shift)},{\y+.23*sin(\a+\shift)}) circle (1.25pt);
  }
  \node[align=left,anchor=west] at (2.65,.80)
    {equally spaced points\\on occupied parallels};
  \draw[->] (2.61,.62)--(1.72,.57);
  \node[align=left,anchor=west] at (2.65,-1.05)
    {populations increase\\toward the equator};
  \draw[->] (2.61,-.88)--(1.68,-.57);
\end{tikzpicture}
\caption{Schematic deterministic Diamond configuration.  There are $O(\sqrt N)$ occupied
parallels, and the phases of each polygon can be chosen arbitrarily or fixed to any desired standard.}
\label{fig:bemoc}
\end{figure}

\begin{equation}\label{eq:geometry}
 r_j\asymp M\rho_j,\;1\leq j\leq 2M-1,\quad  \text{ and  } \quad \sum_{j=1}^{2M-1}\rho_j\asymp M.                    
\end{equation}

\section{Proof of Theorem \ref{thm:main}}


We will use the following  well known integral formula, valid for any integrable function $f:\mathbb S^2\to\mathbb R$:
\begin{equation}\label{eq:integralaintervalo}\int_{\mathbb S^2}f(x)d\sigma(x)=\frac{1}{2}\int_{-1}^1 \frac{1}{2\pi}\int_0^{2\pi} f(\sqrt{1-t^2}\cos \theta,\sqrt{1-t^2}\sin \theta,t)d\theta dt.
\end{equation}
We need to bound the LHS in \eqref{theo:ineq}, which is the error between the Riesz energy of the measure $N\sigma$ and the Riesz energy of the fully discrete measure $\sum_{x\in\mathcal P_N}\delta_x.$

For the parallel with height $z$ we denote as $U_z$ the uniform probability measure on the circle at height $z$ and define $\nu=\sum_jr_jU_{h_j}.$ Then its energy is
\begin{equation}\label{eq:energianu}E_\alpha[\nu]=\int_{\Sph}\int_{\Sph}|x-y|^\alpha d\nu(x)d\nu(y)=\sum_{j,k}r_jr_k
F_\alpha(h_j,h_k)\end{equation}
where
\begin{equation}\label{eq:Falpha}
 F_\alpha(s,t)=\int_{\Sph}\int_{\Sph}|x-y|^\alpha d U_{s}(x)d U_{t}(y)=\frac1{2\pi}\int_0^{2\pi}
 \left(2-2st-2\sqrt{1-s^2}\sqrt{1-t^2}\cos\theta
 \right)^{\alpha/2}d\theta.
\end{equation}

Given any two (possibly equal) parallels $Q_j$ and $Q_k$ in our construction, write
$$\Sigma_{\alpha,jk}=
\sum_{x \in \mathcal P_N\cap Q_{j}}\sum_{y \in \mathcal P_N\cap Q_{k}}|x-y|^\alpha,\;\;\mbox{so that}\;\;
 \sum_{x,y\in \mathcal{P}_N}|x-y|^\alpha
 =\sum_{j,k}\Sigma_{\alpha,jk}.$$
We then have the decomposition
\begin{equation}\label{eq:decomp}
\text{LHS in \eqref{theo:ineq}}=I_\alpha N^2-\sum_{x,y\in \mathcal{P}_N}|x-y|^\alpha
 =A_{\alpha,N}+B_{\alpha,N},             
\end{equation}
where
\begin{align*}
 A_{\alpha,N}&=I_\alpha N^2-E_\alpha[\nu],\\
 B_{\alpha,N}&=\sum_{j,k}\left(r_jr_kF_\alpha(h_j,h_k)
                    -\Sigma_{\alpha,jk}\right).
\end{align*}

\noindent Observe that $|x-y|^\alpha$ is strictly
conditionally negative definite for $0<\alpha<2$ (see for example \cite[Th. 4.4.7]{BHS19}).  Thus, for every signed
measure $\eta$ of total mass zero,
\[
 \iint |x-y|^\alpha\,d\eta(x)d\eta(y)\le0.
\]
Applying this with $\eta=\nu-N\sigma$ gives $A_{\alpha,N}\ge0$; the identical computation
with the atomic point measure gives that the LHS in \eqref{theo:ineq} is also nonnegative.

In subsequent sections we prove that both $A_{\alpha,N}$ and $B_{\alpha,N}$ in  
the decomposition
(\ref{eq:decomp}) are bounded above by a constant depending on $\alpha$ times $N^{1-\alpha/2}$, see lemmas \ref{lem:BalphaN} and \ref{lem:latitude}, which proves Theorem \ref{thm:main}.

\subsection{The bound of $B_{\alpha,N}$}

In this section, we prove the following result.
\begin{lemma}\label{lem:BalphaN}
For $0<\alpha<2$,
\begin{equation*}
    B_{\alpha,N}\lesssim_\alpha N^{1-\alpha/2}
\end{equation*}
\end{lemma}
We will use the following auxiliary result.
\begin{lemma}\label{lem:trap}
Let $0<\alpha<2$ and $G(\theta)=(A-B\cos\theta)^{\alpha/2}$ with $A\ge B\ge0$.  For all
integers $L\ge1$ and all phases $\phi$,
\[
 \left|\frac1L\sum_{k=0}^{L-1}G(\phi+2\pi k/L)
 -\frac1{2\pi}\int_0^{2\pi}G\right|
 \lesssim_\alpha B^{\alpha/2}L^{-1-\alpha}.             
\]
\end{lemma}


\begin{proof}
See Appendix \ref{appendix1}.
\end{proof}

\begin{proof}[Proof of Lemma \ref{lem:BalphaN}]
    
We apply the triangle inequality and estimate each summand 
\[ \left|r_jr_kF_\alpha(h_j,h_k) -\Sigma_{\alpha,jk}\right| \] 
separately. Recall that at the parallel $Q_j$ there are $r_j$ points of $\mathcal P_N$ lying on a circle of radius $\rho_j$. The angular coordinates of these points can be written as \[ \theta_{j,\ell} = \theta_{j,0}+\frac{2\pi \ell}{r_j}, \qquad \ell=0,\dots,r_j-1, \] for some (unimportant) phase $\theta_{j,0}\in[0,2\pi)$. Define 
\begin{equation*}
G(\theta) = (A-B\cos\theta)^{\alpha/2}, \qquad A=2-2h_jh_k, \qquad B=2\rho_j\rho_k. 
\end{equation*} 
Observe that $A\ge B\ge 0$ and there is  equality $A=B$ if and only if $j=k.$

Then the value of $|x-y|^\alpha$ for two points $x,y\in\mathcal P_N$ lying on two fixed parallels $Q_j$ and $Q_k$ is $G(\theta_{j,\ell}-\theta_{k,\ell'})$ for some $\ell\in\{0,\dots,r_j-1\}$ and $\ell'\in\{0,\dots,r_k-1\}$. Therefore, 
\begin{equation*} \Sigma_{\alpha,jk} = \sum_{\ell=0}^{r_j-1}\sum_{\ell'=0}^{r_k-1} G(\theta_{j,\ell}-\theta_{k,\ell'}) 
= \sum_{\ell=0}^{r_j-1}\sum_{\ell'=0}^{r_k-1} G\left( \theta_{j,0}-\theta_{k,0} + 2\pi\left(\frac{\ell}{r_j}-\frac{\ell'}{r_k}\right) \right). 
\end{equation*}
Let 
\[ L=\mathrm{lcm}(r_j,r_k) \qquad\text{and}\qquad d=\mathrm{gcd}(r_j,r_k) \]
denote, respectively, the least common multiple and the greatest common divisor of $r_j$ and $r_k$. Among the $r_jr_k$ angles 
$2\pi\left(\frac{\ell}{r_j}-\frac{\ell'}{r_k}\right)$
appearing in the sum above, only $L$ are distinct modulo $2\pi$, and each of them occurs exactly $d$ times. Hence, 
\[ 
\Sigma_{\alpha,jk} = d\sum_{u=0}^{L-1} G\left(\phi+\frac{2\pi u}{L}\right), \qquad \phi=\theta_{j,0}-\theta_{k,0}. \]
Since $r_jr_k=dL$ and 
\[ F_\alpha(h_j,h_k) = \frac{1}{2\pi}\int_0^{2\pi}G(\theta)\,d\theta, \] 
subsequently using Lemma~\ref{lem:trap} and the relation $r_j\asymp M\rho_j$ from \eqref{eq:geometry}, we finally obtain 

\begin{align*} 
|r_jr_kF_\alpha(h_j,h_k) -\Sigma_{\alpha,jk}|  =& dL\left| \frac{1}{2\pi}\int_0^{2\pi}G(\theta)\,d\theta - \frac{1}{L}\sum_{u=0}^{L-1} G\left(\phi+\frac{2\pi u}{L}\right) \right|\\\nonumber
    \lesssim_\alpha&
        (\rho_j\rho_k)^{\alpha/2}
    \frac{d^{1+\alpha}}{(r_jr_k)^\alpha} 
    \lesssim_\alpha 
     M^{-\alpha}
    \frac{d^{1+\alpha}}
         {(r_jr_k)^{\alpha/2}}\nonumber.
\end{align*}

Trivially observe that for any positive integer $n$, there exist at most $3$ occupied parallels that have exactly $n$ points (one in the northern hemisphere, another one in the southern hemisphere, and eventually the equator). Hence, since $r_M\leq 15M$,
\begin{align*}
    |B_{\alpha,N}|
    \lesssim_\alpha
    M^{-\alpha}
    \sum_{1\leq u,v\leq 15M}
    \frac{\mathrm{gcd}(u,v)^{1+\alpha}}
         {(uv)^{\alpha/2}}.
\end{align*}
It remains to estimate the last double sum. For every pair $u,v$, write
uniquely
\[
    u=da,\qquad v=db,\qquad \mathrm{gcd}(a,b)=1,
\]
where $d=\mathrm{gcd}(u,v)$. Then
\[
    \frac{\mathrm{gcd}(u,v)^{1+\alpha}}{(uv)^{\alpha/2}}
    =
    \frac{d}{(ab)^{\alpha/2}}.
\]
Therefore
\begin{align*}
    \sum_{1\leq u,v\leq 15M} &
    \frac{\mathrm{gcd}(u,v)^{1+\alpha}}{(uv)^{\alpha/2}}
    =
    \sum_{\substack{a,b\geq1\\\mathrm{gcd}(a,b)=1}}
    \frac{1}{(ab)^{\alpha/2}}
    \sum_{1\leq g\leq 15M/\max(a,b)}g
    \\
    \lesssim & M^2 \sum_{\substack{a,b\geq1\\\mathrm{gcd}(a,b)=1}}
    \frac{1}{(ab)^{\alpha/2}\max(a,b)^2}
    \lesssim M^2\sum_{a,b\geq1}
    \frac{1}{(ab)^{1+\alpha/2}}
    =M^2\zeta\left(\frac{\alpha}{2}+1\right)^2
    \lesssim_\alpha M^2.
\end{align*}

Restoring the factor $M^{-\alpha}$ gives
\[
    |B_{\alpha,N}|
    \lesssim_\alpha M^{2-\alpha}\asymp N^{1-\alpha/2},
\]
as claimed.

\end{proof}



\subsection{Bound of $A_{\alpha,N}$}

We will compare the union of rings measure $\nu=\sum_j r_jU_{h_j}$ with $N\sigma$  after projecting both
measures onto the height variable. Since the normalized surface measure on
$\mathbb S^2$ has height density $dt/2$, define
\[
 \lambda=\frac N2\mathbf 1_{[-1,1]}(t)\,dt,
 \qquad
 \nu_z=\sum_j r_j\delta_{h_j}.
\]
Thus $\lambda$ is the height projection of $N\sigma$, whereas $\nu_z$ is the
height projection of $\nu$.

For $1\le j\le 2M-1$, let
\[
 B_j=[H_j,H_{j-1}],
\]
then, from \eqref{eq:integralaintervalo}, $\lambda(B_j)=N\sigma(\mathcal B_j)$.
Define the following measures on $B_j$
\begin{align*}
 \nu_j={}&r_j\delta_{h_j},\quad\Rightarrow \nu_z=\sum\nu_j,\\
 \lambda_j={}&\frac N2\mathbf 1_{B_j}(t)\,dt,\quad\Rightarrow \lambda=\sum\lambda_j,\\
 \mu_j={}&\lambda_j-\nu_j.
\end{align*}
We next rewrite $A_{\alpha,N}$ as a sum over the bands. First, note from \eqref{eq:integralaintervalo} that
\begin{align*}
    I_\alpha=E_\alpha[\sigma]=&\int_{\Sph}\int_{\Sph}|x-y|^\alpha d \sigma(x)d \sigma(y)\\
    =&\frac1{N^2}\int_{-1}^1\int_{-1}^1F_\alpha(s,t)\,d\lambda(t)\,d\lambda(s)\\
=&\frac1{N^2}\sum_{j,k=1}^{2M-1}\int_{B_j}\int_{B_k}F_\alpha(s,t)\,d\lambda_k(t)\,d\lambda_j(s)
\end{align*}
We also have from \eqref{eq:energianu}:
\begin{equation*}
 E_\alpha[\nu]=\sum_{j,k=1}^{2M-1}r_jr_k F_\alpha(h_j,h_k)
=\sum_{j,k=1}^{2M-1}\int_{B_j}\int_{B_k} F_\alpha(s,t)\,d\nu_k(t)\,d\nu_j(s).
\end{equation*}
Rotational invariance gives
\[
 \int_{-1}^1F_\alpha(s,t)\,d\lambda(t)=NI_\alpha
 \qquad(-1\le s\le1).
\]
Thus the mixed term in the energy expansion vanishes for
$\mu:=\lambda-\nu_z=\sum_j\mu_j$, which has total mass zero.
We thus have
\begin{equation}\label{eq:latitude-decomposition}
 A_{\alpha,N}
 =I_\alpha N^2-E_\alpha[\nu]
 =-\sum_{j,k=1}^{2M-1}
   \int_{B_j}\int_{B_k}F_\alpha(s,t)\,d\mu_k(s)d\mu_j(t).
\end{equation}

We will use the following technical lemma.

\begin{lemma}\label{lem:blocks}
Let $0<\alpha<2$. There exists a constant $C_\alpha$ with the following properties.

If the two bands have comparable scales,
\[
 \frac{r_j}{8}\le r_k\le 8 r_j,
\]
then
\begin{equation}\label{eq:comparableblock}
 \left|
 \int_{B_j}\int_{B_k}F_\alpha(s,t)\,d\mu_k(s)d\mu_j(t)
 \right|
 \le C_\alpha\frac{r_j}{M^\alpha}
 (1+|j-k|)^{\alpha-3}.                               
\end{equation}

If $r_k> 8 r_j$ and the two bands are on the same side of the equator
(that is, either $j,k<M$ or $j,k>M$), then
\begin{equation}\label{eq:farblock}
 \left|
 \int_{B_j}\int_{B_k}F_\alpha(s,t)\,d\mu_k(s)d\mu_j(t)
 \right|
 \le C_\alpha\frac{r_j^3}{M^\alpha r_k^{5-\alpha}}.
\end{equation}

If $r_k>8r_j$ and the bands are not on the same side
of the equator, including the case in which one of the indices is $M$,
then
\begin{equation}\label{eq:oppositeblock}
 \left|
 \int_{B_j}\int_{B_k}F_\alpha(s,t)\,d\mu_k(s)d\mu_j(t)
 \right|
 \le C_\alpha\frac{r_j^3r_k^3}{M^8}.                
\end{equation}
\end{lemma}
\begin{proof}
    See Appendix \ref{appendix2}.
\end{proof}

We are now ready to prove the desired bound.
\begin{lemma}\label{lem:latitude}
For $0<\alpha<2$,
\[
 A_{\alpha,N}\lesssim_\alpha N^{1-\alpha/2}.
\]
\end{lemma}

\begin{proof}
Taking absolute values in
\eqref{eq:latitude-decomposition}, we divide the pairs $(j,k)$ according to
Lemma~\ref{lem:blocks}.

For comparable scales, \eqref{eq:comparableblock} gives
\begin{align*}
 \sum_{\substack{1\le j,k\le2M-1\\r_j/8\le r_k\le8r_j}}
 \left|
 \int_{B_j}\int_{B_k}F_\alpha\,d\mu_jd\mu_k
 \right|
 \lesssim_\alpha&
 M^{-\alpha}
 \sum_{j=1}^{2M-1}r_j
 \sum_{k=1}^{2M-1}(1+|j-k|)^{\alpha-3}\\
 \lesssim_\alpha&
 M^{-\alpha}
 \sum_{j=1}^{2M-1}r_j\lesssim M^{2-\alpha}\asymp N^{1-\alpha/2},
\end{align*}
the second inequality because,
since $\alpha<2$, the inner sum is at most $2\zeta(3-\alpha)<\infty$.

For unequal scales on the same side of the equator, it is enough by
symmetry to sum for $r_k>8r_j$. Since on either side there is at most one band
of each scale, \eqref{eq:farblock} gives
\[
 \sum_{r_k>8r_j}
 \left|
 \int_{B_j}\int_{B_k}F_\alpha\,d\mu_jd\mu_k
 \right|
 \lesssim_\alpha M^{-\alpha}
 \sum_{r=1}^{M}\frac1{r^{5-\alpha}}
 \sum_{m<r/8}m^3
 \lesssim_\alpha
 M^{-\alpha}
 \sum_{r=1}^{M}{r^{\alpha-1}}\lesssim_\alpha 1.
\]
The reverse orientation and the other
hemisphere give the same estimate.

Finally, \eqref{eq:oppositeblock} gives for all remaining noncomparable
pairs
\[
 \sum_{j,k}
 \left|
 \int_{B_j}\int_{B_k}F_\alpha\,d\mu_jd\mu_k
 \right|
 \lesssim_\alpha M^{-8}
 \left(\sum_{j=1}^{2M-1}r_j^3\right)^2\\
 \lesssim 1.
\]
Combining the three estimates we get $
 A_{\alpha,N}\le C_\alpha N^{1-\alpha/2}$, as wanted.

\end{proof}

\appendix

\section{Proof of Lemma \ref{lem:trap}}\label{appendix1}

The assertion is trivial when $B=0$. Assume therefore that $B>0$, and set
\[
    a=\frac{\alpha}{2}\in(0,1),
    \qquad
    \delta=\frac{A-B}{B}\geq 0.
\]
Then
\[
    G(\theta)=B^a f_\delta(\theta),
    \;\;\mbox{with}\;\;
    f_\delta(\theta):=(1+\delta-\cos\theta)^a.
\]
We use the Fourier normalization
\[
    \widehat h(n)
    :=
    \frac{1}{2\pi}\int_0^{2\pi}
    h(\theta)e^{-in\theta}\,d\theta.
\]

We first prove that, for every $n\neq 0$,
\begin{equation}\label{eq:monotone-fourier}
    \widehat f_0(n)
    \leq
    \widehat f_\delta(n)
    \leq 0.
\end{equation}

Since $f_\delta$ is real and even $\widehat f_\delta(-n)=\widehat f_\delta(n),$
so it suffices to consider $n\geq 1$.

We begin with an elementary positivity observation for $\delta>0$. In this range,
\begin{equation*}
    (1+\delta-\cos\theta)^{a-1}
    =(1+\delta)^{a-1}\left(1-\frac{\cos\theta}{1+\delta}\right)^{a-1}  
    =(1+\delta)^{a-1}
      \sum_{k=0}^{\infty}
      \frac{(1-a)_k}{k!}  \left(\frac{\cos\theta}{1+\delta}\right)^k,
\end{equation*}
and the binomial series converges absolutely and uniformly in $\theta$. Moreover, as
\[
    \cos^k\theta
    =
    2^{-k}\sum_{j=0}^k
    \binom{k}{j}e^{i(k-2j)\theta}.
\]
for $n\geq0$,
\[
    \widehat{\cos^k}(n)
    =
    \begin{cases}
    \displaystyle
    2^{-k}\binom{k}{(k-n)/2},
        & k\geq n \ \text{and}\ k\equiv n \pmod 2,\\[2ex]
    0,  & \text{otherwise}.
    \end{cases}
\]
All these coefficients are nonnegative, and the term $k=n$ is strictly
positive. It follows that
\begin{equation*}
    \widehat g_\delta(n)>0,
    \qquad
    g_\delta(\theta):=(1+\delta-\cos\theta)^{a-1},
    \qquad
    \delta>0,\ n\in\mathbb Z.
\end{equation*}

For every fixed $\delta_0>0$, differentiation under the integral sign is
justified by dominated convergence
because the derivative is bounded
$$\left. \frac{d}{d\delta}\right|_{\delta=\delta_0}f_\delta(\theta)=ag_{\delta_0}(\theta)\le a{\delta_0}^{a-1}.$$
Therefore
\[
    \frac{d}{d\delta}\widehat f_\delta(n)
    =
    a\,\widehat g_\delta(n)>0,
    \qquad
    \delta>0,\quad n\geq1.
\]
Thus, for fixed $n\ge 1,$ $\delta\mapsto\widehat f_\delta(n)$ is strictly increasing on
$(0,\infty)$, and
$\|f_\delta-f_0\|_\infty\le\delta^a\to0$ as $\delta\downarrow0$, proving
\begin{equation*}
    \widehat f_0(n)\leq\widehat f_\delta(n),
    \qquad
    \delta>0,\quad n\geq1.
\end{equation*}

To get (\ref{eq:monotone-fourier}) it remains to prove that the increasing function
$\widehat f_\delta(n)$ converges to zero as $\delta\to +\infty$. 
Since the constant function has zero $n$-th Fourier coefficient for
$n\neq0$,
\[
    \widehat f_\delta(n)
    =     \frac{1}{2\pi}\int_0^{2\pi}
      \bigl((1+\delta-\cos\theta)^a-(1+\delta)^a\bigr)e^{-in\theta}\,d\theta.
\]
Since the absolute value of the integrand is bounded above uniformly by $a\delta^{a-1}$ (use the Mean Value Theorem), which is bounded above for large $\delta$, we can use Lebesgue's Dominated Convergence Theorem to conclude that, in particular,
\begin{equation}\label{eq:fourier-domination}
    |\widehat f_\delta(n)|
    \leq
    |\widehat f_0(n)|,
    \qquad n\neq0,\quad \delta\geq0.
\end{equation}

We now compute the coefficients of $f_0$. Since
\[
    f_0(\theta)
    =(1-\cos\theta)^a
    =2^a\sin^{2a}\frac{\theta}{2},
\]
we have, for $n\geq1$,
\[
    \widehat f_0(n)
    =
    \frac{2^a}{\pi}
    \int_0^\pi
    \sin^{2a}u\,\cos(2nu)\,du.
\]
The standard beta-integral identity \cite[3.631 (8)]{GR07}
\[
    \int_0^\pi
    \sin^{2a}u\,\cos(2nu)\,du
    =
    \frac{(-1)^n\pi\Gamma(2a+1)}
    {2^{2a}\Gamma(a-n+1)\Gamma(a+n+1)}
\]
and the Gamma function reflection formula $\Gamma(z)\Gamma(1-z)=\pi/\sin(\pi z)$ (valid for noninteger $z$) give
\begin{equation}\label{eq:f0-explicit}
    \widehat f_0(n)
    =
    -\frac{2^{-a}\Gamma(2a+1)\sin(\pi a)}{\pi}
    \frac{\Gamma(n-a)}{\Gamma(n+a+1)}<0.
\end{equation}
The elementary estimate \[
    \frac{\Gamma(n-a)}{\Gamma(n+a+1)}
    \leq C_a n^{-1-2a}
\]
combined with \eqref{eq:fourier-domination} and
\eqref{eq:f0-explicit} gives:
\[
    |\widehat f_\delta(n)|
    \leq C_a |n|^{-1-2a}
    =
    C_\alpha |n|^{-1-\alpha},
    \qquad n\neq0,
\]
uniformly for $\delta\geq0$. Therefore
\[
    |\widehat G(n)|
    \leq
    C_\alpha B^{\alpha/2}|n|^{-1-\alpha},
    \qquad n\neq0.
\]
Since the Fourier coefficients are absolutely summable, the Fourier
series of $G$ converges absolutely and uniformly. Thus the discrete Poisson formula (see for example \cite[eq. (3.17)]{TW14})
gives
\[
    \frac1L\sum_{k=0}^{L-1}
    G\left(\phi+\frac{2\pi k}{L}\right)
    -
    \frac1{2\pi}\int_0^{2\pi}G(\theta)\,d\theta
    =
    \sum_{m\neq0}
    \widehat G(mL)e^{imL\phi}.
\]
Consequently,
\begin{equation*}
    \left|
    \frac1L\sum_{k=0}^{L-1}
    G\left(\phi+\frac{2\pi k}{L}\right)
    -
    \frac1{2\pi}\int_0^{2\pi}G(\theta)\,d\theta
    \right|
    \leq
    C_\alpha B^{\alpha/2}
    \sum_{m\neq0}|mL|^{-1-\alpha}\leq
    C_\alpha B^{\alpha/2}L^{-1-\alpha}.
\end{equation*}

This proves the lemma.
\section{Proof of Lemma \ref{lem:blocks}}\label{appendix2}

If $J_j:=\{\arccos t:t\in B_j\}$ is the corresponding interval of polar
angles, then
\begin{equation}\label{eq:angular-band-scales}
 \mathrm{Length}(J_j)\asymp \frac1M,\;\mbox{for}\;1\le j\le 2M-1
 \;\;\mbox{and}\;\;
 \sin\phi\asymp\frac{r_j}{M}
 \quad(\phi\in J_j,\ j\notin\{1,2M-1\}).    
\end{equation}
On the two polar intervals we have only
$\sin \phi\lesssim M^{-1}$
for $\phi\in J_1\cup J_{2M-1}.$
Those two endpoint
bands will always be covered by the elementary estimates below. Moreover,
when $|j-k|\ge2$,
\begin{equation}\label{eq:angular-separation-bands}
 \operatorname{dist}(J_j,J_k)\asymp\frac{|j-k|}{M},
\end{equation}
for all $1\le j,k\le 2M-1.$
These estimates follow from $N\asymp M^2$,
and the mean value theorem applied to $\arccos t$.

Recall the definition of $F_\alpha$ from \eqref{eq:Falpha}. Since the integrand is symmetric about $\theta=\pi$, it is convenient to
write
\begin{equation}\label{eq:F-polar-average}
 F_\alpha(s,t)=\frac1\pi\int_0^\pi D(s,t,\theta)^\alpha\,d\theta,
\end{equation}
where, with $\phi=\arccos s$ and $\psi=\arccos t$,
\begin{align}
 D(s,t,\theta)^2
 &=2-2\bigl(\cos\phi\cos\psi+
              \sin\phi\sin\psi\cos\theta\bigr)\notag\\
 &=2(1-\cos(\phi-\psi))
   +2\sin\phi\sin\psi(1-\cos\theta).              
\end{align}
In particular, for $0\le\theta\le\pi$,
\begin{equation}\label{eq:distance-comparison}
 D(s,t,\theta)^2
 \asymp (\phi-\psi)^2+
          \sin\phi\sin\psi\,\theta^2.              
\end{equation}

Note that for all $j$,
\[
 \lambda_j(B_j)=\frac N2(H_{j-1}-H_j)=r_j=\nu_j(B_j),\quad \int_{B_j}t\,d\lambda_j(t)
 =\frac N4(H_{j-1}^2-H_j^2)
 =r_jh_j= \int_{B_j}t\,d\nu_j(t).
\]
Therefore for $\mu_j=\lambda_j-\nu_j$
\begin{equation*}
 \int_{B_j}d\mu_j=0,
 \qquad
 \int_{B_j}t\,d\mu_j(t)=0,
\end{equation*}
and 
\begin{equation*}
 \|\mu_j\|_{\rm TV}
 \le \|\nu_j\|_{\rm TV}+\|\lambda_j\|_{\rm TV}
 =2r_j,\quad j=1,\ldots,2M-1,
\end{equation*}
where $\|\eta\|_{\rm TV}$ denotes the total variation of a signed
measure $\eta$.
We shall repeatedly use the following consequence of Taylor's theorem. If
a signed measure $\eta$ is supported on an interval $I$ of length $a$ and
satisfies
\[
 \int_I d\eta=\int_I t\,d\eta(t)=0,
\]
then, for every $f\in C^2(I)$,
\begin{equation}\label{eq:one-variable-taylor}
 \left|\int_I f\,d\eta\right|
 \le \frac12\|\eta\|_{\rm TV}a^2\sup_I|f''|.        
\end{equation}
Indeed, subtract from $f$ its affine Taylor polynomial at any point of
$I$ and use the remainder formula. Applying \eqref{eq:one-variable-taylor}
 two times and using $H_{j-1}-H_j\lesssim r_j/M^2$ gives for a function $G\in C^4(B_j\times B_k)$
\begin{align}
 \left|
 \int_{B_k}\int_{B_j}G(s,t)\,d\mu_j(s)d\mu_k(t)
 \right|
 \le& \frac{r_k^2}{2M^4}\|\mu_k\|_{\rm TV}\sup_{t\in B_k}\left|\frac{\partial^2}{\partial t^2}\int_{B_j}G(s,t)\,d\mu_j(s)\right|\label{eq:two-variable-taylor}\\
 \nonumber
  =& \frac{r_k^2}{2M^4}\|\mu_k\|_{\rm TV}\sup_{t\in B_k}\left|\int_{B_j}\frac{\partial^2}{\partial t^2}G(s,t)\,d\mu_j(s)\right|\\
  \nonumber
 \le&\frac{r_j^2r_k^2}{4M^8}\|\mu_j\|_{\rm TV}\|\mu_k\|_{\rm TV}
 \sup_{B_j\times B_k}|\partial_s^2\partial_t^2G|   \\
 \nonumber
 \le&\frac{r_j^3r_k^3}{M^8}
 \sup_{B_j\times B_k}|\partial_s^2\partial_t^2G|,
\end{align}
where the continuity of the mixed derivative on $B_j\times B_k$ justifies the interchange of derivative and integral.

We now record two elementary derivative estimates.

\begin{lemma}
\label{lem:derivative-bounds}
Let for $-1\le s,t\le 1$ and $0\le \theta\le 2\pi$
\[
D(s,t,\theta)
=
\left|
\bigl(\sqrt{1-s^2}\cos\theta,
      \sqrt{1-s^2}\sin\theta,s\bigr)
-
\bigl(\sqrt{1-t^2},0,t\bigr)
\right|,
\]
so that
\[
D(s,t,\theta)^2
=
2-2\left(
st+\sqrt{1-s^2}\sqrt{1-t^2}\cos\theta
\right).
\]
Writing
\[
s=\cos\phi,\qquad t=\cos\psi,
\]
we equivalently have
\[
D(s,t,\theta)
=
|x(\phi,\theta)-y(\psi)|,
\]
where
\[
x(\phi,\theta)
=
(\sin\phi\cos\theta,\sin\phi\sin\theta,\cos\phi),
\qquad
y(\psi)
=
(\sin\psi,0,\cos\psi).
\]

At every point for which $D(s,t,\theta)>0$, and for all
$m,n\ge0$ with $m+n\le4$,
\begin{equation}
\label{eq:polar-derivatives}
\left|
\partial_\phi^m\partial_\psi^n
D(s,t,\theta)^\alpha
\right|
\le
C_\alpha
D(s,t,\theta)^{\alpha-m-n}.
\end{equation}
Moreover, if $-1<s,t<1$,
\begin{equation}\label{eq:height-derivatives}
\left|
\partial_s^2\partial_t^2
D(s,t,\theta)^\alpha
\right|
\le C_\alpha\left(
\frac{D^{\alpha-4}}{(\sin\phi\sin\psi)^{2}}
+\frac{D^{\alpha-3}}{\sin^{2} \phi \sin^{3} \psi}
+\frac{D^{\alpha-3}}{\sin^3 \phi \sin^2 \psi}
+\frac{D^{\alpha-2}}{(\sin\phi\sin\psi)^{3}}
\right),
\end{equation}
where we write $D=D(s,t,\theta).$
\end{lemma}

\begin{proof}
Set $z(\phi,\psi,\theta)
=
x(\phi,\theta)-y(\psi),$
so that
\[
D(s,t,\theta)^\alpha=|z(\phi,\psi,\theta)|^\alpha.
\]
For
\[
F(z)=|z|^\alpha,\qquad z\neq0,
\]
the $q$-th differential $D^qF$ is homogeneous of degree $\alpha-q$.
Hence
\[
|D^qF(z)|
\le
C_{\alpha,q}|z|^{\alpha-q}.
\]
Observe that every derivative of $z$ with respect to $\phi$ and $\psi$ is uniformly
bounded. Therefore, if $r=m+n$, repeated application of the chain rule
expresses
\[
\partial_\phi^m\partial_\psi^nF(z)
\]
as a finite sum of terms involving $D^qF(z)$, with $1\le q\le r$,
applied to bounded derivatives of $z$. Each such term is bounded by
\[
C_\alpha D^{\alpha-q}.
\]
Since $D\le2$ and $q\le r$,
\[
D^{\alpha-q}
=
D^{\alpha-r}D^{r-q}
\le
2^{r-q}D^{\alpha-r}.
\]
This proves
\[
\left|
\partial_\phi^m\partial_\psi^nD^\alpha
\right|
\le
C_\alpha D^{\alpha-m-n}.
\]

We now pass to the height variables. Since
\[
s=\cos\phi,\qquad t=\cos\psi,
\]
we have
\[
\partial_s
=
-\frac1{\sin\phi}\partial_\phi,
\qquad
\partial_t
=
-\frac1{\sin\psi}\partial_\psi,
\]
and therefore
\[
\partial_s^2
=
\frac1{\sin^2\phi}\partial_\phi^2
-
\frac{\cos\phi}{\sin^3\phi}\partial_\phi,
\]
with the analogous formula for $\partial_t^2$.

Expanding $\partial_s^2\partial_t^2D^\alpha$ gives four terms
\begin{equation*}
\partial_s^2\partial_t^2D^\alpha
=
\frac{\partial_\phi^2\partial_\psi^2D^\alpha}
     {\sin^2\phi\,\sin^2\psi}
-
\frac{\cos\psi\,
      \partial_\phi^2\partial_\psi D^\alpha}
     {\sin^2\phi\,\sin^3\psi}-
\frac{\cos\phi\,
      \partial_\phi\partial_\psi^2D^\alpha}
     {\sin^3\phi\,\sin^2\psi}
+
\frac{\cos\phi\cos\psi\,
      \partial_\phi\partial_\psi D^\alpha}
     {\sin^3\phi\,\sin^3\psi}.
\end{equation*}
Using $|\cos\phi|,|\cos\psi|\le1$ and the preceding polar derivative
estimate with total orders $4,3,3,$ and $2$, respectively, yields
\eqref{eq:height-derivatives}.
\end{proof}




\begin{lemma}
\label{lem:separated-derivative}
Let for $-1\le s,t\le 1$
$$U=2-2st,\qquad V=2\sqrt{1-s^2}\sqrt{1-t^2},$$
and suppose that, for some fixed $\varepsilon>0$,
$U>0$ and
$V\le (1-\varepsilon)U.$ Then
\begin{equation}
\label{eq:separated-derivative}
\left|
\partial_s^2\partial_t^2F_\alpha(s,t)
\right|
\le
C_{\alpha,\varepsilon}\,
U^{\alpha/2-4}.
\end{equation}
At boundary points satisfying the hypotheses, derivatives are
understood via the smooth extension defined by the series below.
\end{lemma}

\begin{proof}
Recall that
\[
F_\alpha(s,t)
=
\frac1\pi\int_0^\pi
\bigl(U-V\cos\theta\bigr)^{\alpha/2}\,d\theta .
\]
Since $\frac{V}{U}\le1-\varepsilon,$
the binomial series converges uniformly, and we may integrate it term by
term. Since the odd powers of $\cos\theta$ have zero integral, we obtain
\[
F_\alpha(s,t)
=
U^{\alpha/2}
\sum_{m=0}^\infty
a_m
\left(\frac{V}{U}\right)^{2m},
\]
where
\[
a_m
=
\binom{\alpha/2}{2m}c_m,
\qquad
c_m
=
\frac1\pi\int_0^\pi\cos^{2m}\theta\,d\theta.
\]
Observe that the coefficients $a_m$ are easily shown to be uniformly bounded in $m$. We write the $m$-th term of the sum above as
\[
T_m(s,t)
=
a_m4^m
(1-s^2)^m(1-t^2)^m
U^{\alpha/2-2m}.
\]
Every term produced by applying
$\partial_s^2\partial_t^2$ to $T_m$ is bounded in absolute value by an expression of the
form
\[
C_\alpha(1+m)^4
4^m
(1-s^2)^{m-a}
(1-t^2)^{m-b}
U^{\alpha/2-2m-c},
\]
where
\[
0\le a\le2,\qquad
0\le b\le2,\qquad
a+b+c\le4.
\]
Indeed, the factor $(1+m)^4$ accounts for the coefficients arising when
the three powers are differentiated a total of four times. Notice also
that all derivatives of $U=2-2st$
which occur here are uniformly bounded. Set
\[
A=1-s^2,\qquad B=1-t^2.
\]
We have
\[
A\le U,\qquad B\le U,
\]
because, for instance,
\[
U-A
=
1+s^2-2st
=
(s-t)^2+1-t^2
\ge0,
\]
and similarly for $B$. For $m\ge2$,
\begin{align*}
4^mA^{m-a}B^{m-b}
&=
16(4AB)^{m-2}A^{2-a}B^{2-b}=
16V^{2m-4}A^{2-a}B^{2-b}\\
&\le
C(1-\varepsilon)^{2m-4}
U^{2m-4}U^{4-a-b}=
C(1-\varepsilon)^{2m-4}
U^{2m-a-b}.
\end{align*}
Hence
\begin{align*}
\left|
\partial_s^2\partial_t^2T_m(s,t)
\right|
&\le
C_\alpha(1+m)^4
(1-\varepsilon)^{2m-4}
U^{\alpha/2-a-b-c}\\
&\le
C_\alpha(1+m)^4
(1-\varepsilon)^{2m-4}
U^{\alpha/2-4},
\end{align*}
where in the last step we used $a+b+c\le4$ and $U\le4$.

Therefore
\[
\sum_{m=2}^{\infty}
\left|
\partial_s^2\partial_t^2T_m(s,t)
\right|
\le
C_{\alpha,\varepsilon}U^{\alpha/2-4},
\]
since
\[
\sum_{m=2}^\infty
(1+m)^4(1-\varepsilon)^{2m-4}<\infty.
\]
This also justifies term-by-term differentiation of the series.
The same argument gives locally uniform convergence for all
partial derivatives of total order at most four, since $U$ is locally
bounded below. At boundary points, the power series in
$4(1-s^2)(1-t^2)/U^2$ defines a smooth extension to a neighborhood,
where the absolute value of this ratio remains less than one.

It remains only to consider $m=0$ and $m=1$. These terms are
\[
T_0=a_0U^{\alpha/2},
\qquad
T_1=4a_1AB\,U^{\alpha/2-2}.
\]
A direct differentiation, using again
$A,B\le U$ and $U\le4,$
gives
\[
\left|\partial_s^2\partial_t^2T_0\right|
+
\left|\partial_s^2\partial_t^2T_1\right|
\le
C_\alpha U^{\alpha/2-4}.
\]
Combining this with the estimate for $m\ge2$ proves
\eqref{eq:separated-derivative}.
\end{proof}

We now prove \eqref{eq:comparableblock}, so we assume that $r_j/8\leq r_k\leq 8r_j$. Put
\[
 \ell=|j-k|,
 \qquad
 r=r_j\asymp r_k.
\]
\noindent We distinguish several cases:
\begin{itemize}
    \item $\{j,k\}\cap\{1,2M-1\}\ne\varnothing$. By interchanging indices and reflecting across the equator, it suffices to take $j=1$. Then $r_j=4$, $r_k\le32$, and $r_M\ge4M\ge64$, so either $k\le8$ or $2M-k\le8$.
In the first case both bands lie near the north pole, and
$\phi+\psi\lesssim M^{-1}$.
Hence,
\[
D(s,t,\theta)^2\underset{\eqref{eq:distance-comparison}}\lesssim (\phi-\psi)^2+\phi\psi\lesssim (\phi+\psi)^2\lesssim M^{-2}
\underset{\eqref{eq:F-polar-average}}{\Rightarrow}\sup_{B_j\times B_k}F_\alpha(s,t)\lesssim_\alpha M^{-\alpha},
\]
and therefore, using \(\|\mu_j\|_{\rm TV}\le 2r_j\) and the same bound for $k$, we have proved:
\[
\left|
\int_{B_j}\int_{B_k}
F_\alpha(s,t)\,d\mu_k(t)d\mu_j(s)
\right|
\lesssim_\alpha M^{-\alpha},
\]
which implies \eqref{eq:comparableblock}.
In the second case the bands lie near opposite poles. On the whole rectangle,
$s\ge H_1=1-8/N\ge127/128$ and $t\le0$. Thus $2\le U\le4$ and
$V/U\le\sqrt{1-s^2}\le1/2$.
Lemma~\ref{lem:separated-derivative} and \eqref{eq:two-variable-taylor} give
\[
 \left|\iint F_\alpha\,d\mu_jd\mu_k\right|
 \lesssim_\alpha \frac{r_j^3r_k^3}{M^8}
 \lesssim_\alpha M^{-8}.
\]
Here $\ell\asymp M$, so this is stronger than the required bound of order $M^{-3}$.

\item $j,k\not\in\{1,2M-1\}$, so that \eqref{eq:angular-band-scales} gives
$\sin\phi\asymp\sin\psi\asymp r/M$. We consider three subcases:
\begin{itemize}
    \item[$\star$] If $3\le\ell\le  2r$, by \eqref{eq:angular-separation-bands},
\[
 |\phi-\psi|\asymp\frac\ell M\lesssim \frac{r}{M}\underset{\eqref{eq:distance-comparison}}{\Rightarrow} D^2\gtrsim\frac{\ell^2}{M^2}+\frac{r^2}{M^2}\theta^2
\]
Then,
splitting the elementary integrals at $\theta=\ell/r$, one obtains
\begin{align*}
 \int_0^\pi D^{\alpha-4}\,d\theta
 &\lesssim_\alpha \frac{\ell^{\alpha-3}}{rM^{\alpha-4}},\\
 \int_0^\pi D^{\alpha-3}\,d\theta
 &\lesssim_\alpha \frac{\ell^{\alpha-2}}{rM^{\alpha-3}},\\
 \int_0^\pi D^{\alpha-2}\,d\theta
 &\lesssim_\alpha\begin{cases}
 \frac{M}{r}\left(\frac{\ell}{M}\right)^{\alpha-1}&0<\alpha<1\\
 \frac{M}{r}\log\left(2+\frac{r}{\ell}\right)&\alpha=1\\
 \left(\frac{r}{M}\right)^{\alpha-2}&1<\alpha<2
 \end{cases}
\end{align*}
Substitution in \eqref{eq:height-derivatives} and interchange of 
integration and derivative in \eqref{eq:F-polar-average} give
\begin{align}\label{eq:local-mixed-derivative}
 \sup_{B_j\times B_k}
 |\partial_s^2\partial_t^2F_\alpha(s,t)|=&\sup_{B_j\times B_k}\left|\frac1\pi\int_0^\pi \partial_s^2\partial_t^2D(s,t,\theta)^\alpha \,d\theta\right|
  \\
  \lesssim_\alpha&
  \sup_{B_j\times B_k}\int_0^\pi\Biggl(
  \frac{D^{\alpha-4}}{\sin^2\phi\sin^2\psi}
  +\frac{D^{\alpha-3}}{\sin^2\phi\sin^3\psi}\nonumber\\
  &\qquad\qquad\quad
  +\frac{D^{\alpha-3}}{\sin^3\phi\sin^2\psi}
  +\frac{D^{\alpha-2}}{\sin^3\phi\sin^3\psi}
  \Biggr)\,d\theta\nonumber\\
  \lesssim_{\alpha}& \frac{\ell^{\alpha-3}M^8}{r^5M^\alpha}\nonumber.
\end{align}
Indeed, relative to the first term, the other contributions are
bounded by constants times $\ell/r$ and, respectively,
$(\ell/r)^2$, $(\ell/r)^2\log(2+r/\ell)$, or
$(\ell/r)^{3-\alpha}$; all are bounded for $0<\ell/r\le2$.
Combining \eqref{eq:two-variable-taylor} and \eqref{eq:local-mixed-derivative}, we find
\begin{align*}
 \left|
 \int_{B_j}\int_{B_k}F_\alpha\,d\mu_jd\mu_k
 \right|
 \lesssim_\alpha& \frac{r\ell^{\alpha-3}}{M^\alpha},
\end{align*}
proving \eqref{eq:comparableblock}.
\item[$\star$]
If $\ell\le2$, note from \eqref{eq:F-polar-average} and \eqref{eq:distance-comparison} that
\[
\int_0^{1/r}D^\alpha\,d\theta\lesssim r^{-1}M^{-\alpha},
\]
which, together with the bound on the total variation of $\mu_j$ and $\mu_k$, implies
\[
\left| \int_{B_j}\int_{B_k}F_\alpha\,d\mu_jd\mu_k
 \right|
 \lesssim rM^{-\alpha}+
 \left|
\int_{B_j}\int_{B_k}\int_{1/r}^\pi D(s,t,\theta)^\alpha\,d\theta\,d\mu_jd\mu_k
 \right|
\]
With the restriction $1/r\le\theta\le\pi$, the integrand is smooth. Repeating the
calculation leading to \eqref{eq:local-mixed-derivative} gives
\[
 \sup_{B_j\times B_k}
 \left|\partial_s^2\partial_t^2
 \left(\frac1\pi\int_{1/r}^\pi D^\alpha\,d\theta\right)
 \right|
 \lesssim_\alpha \frac{M^8}{r^5M^\alpha}\underset{\eqref{eq:two-variable-taylor}}{\Rightarrow}
 \left|
\int_{B_j}\int_{B_k}\int_{1/r}^\pi D(s,t,\theta)^\alpha\,d\theta\,d\mu_jd\mu_k
 \right|\lesssim_\alpha \frac{r}{M^\alpha}
\]
This proves \eqref{eq:comparableblock} for neighboring bands.

\item[$\star$] If $\ell>2r$, using \eqref{eq:angular-separation-bands} we have the following inequality:
\begin{equation}\label{eq:UmV}
 U-V=2(1-\cos(\phi-\psi))\asymp\frac{\ell^2}{M^2}.
\end{equation}
In particular, $U\gtrsim \ell^2M^{-2}$. Besides, $V=2\sin\phi\sin\psi\lesssim r^2/M^2$. Hence, we have,
\[
U=(U-V)+V\lesssim \frac{\ell^2+r^2}{M^2}\lesssim\frac{\ell^2}{M^2},\text{ and hence }U\asymp \frac{\ell^2}{M^2},
\]
which together with \eqref{eq:UmV} implies $V\le(1-\varepsilon)U$ for some fixed constant $\varepsilon>0$. Hence \eqref{eq:separated-derivative} and
\eqref{eq:two-variable-taylor} yield
\begin{align*}
 \left|
 \int_{B_j}\int_{B_k}F_\alpha\,d\mu_jd\mu_k
 \right|
 &\lesssim_\alpha
 \frac{r^6}{M^8}
 \left(\frac\ell M\right)^{\alpha-8}\lesssim\frac{r\ell^{\alpha-3}}{M^\alpha}.
\end{align*}
Thus \eqref{eq:comparableblock} is also proved  this case.
\end{itemize}
\end{itemize}
We next prove \eqref{eq:farblock}. Suppose, for example, that
$j,k<M$ (the other hemisphere is identical) and $r_k>8r_j$, which readily implies $k>8j\geq3$, $j\leq M/2\leq \sqrt{N}/4$, $1-H_{k-1}\gtrsim (r_k/M)^2$ and $1-H_j\leq k^2/N$. The elementary
band geometry gives
\[
 \qquad
 U\ge 2-2H_{j-1}H_{k-1}\ge 2-2H_{k-1}\gtrsim\left(\frac{r_k}{M}\right)^2.
\]
By writing the value of each term, it is an exercise to prove that
\[
8\left(\left(1-\frac{k^2}{N}\right)-H_{k-1}\right)-\left(1-\left(1-\frac{k^2}{N}\right)H_{k}\right)\geq 0,\quad 
\]
Using that $x\mapsto 8(x-H_{k-1})-(1-xH_{k})$ is an increasing function of $x$, this implies
\[
8(H_j-H_{k-1})-(1-H_{j}H_{k})\geq 0,\quad 
\]
and subsequently
\[
8(s-t)-(1-st)\geq0,\quad s\in B_j,\; t\in B_k,
\]
which is algebraically equivalent to:

\[
V\leq\frac{\sqrt{63}}{8}U.
\]
Therefore, \eqref{eq:two-variable-taylor} and 
\eqref{eq:separated-derivative} yield
\[
 \left|
\int_{B_j}\int_{B_k}F_\alpha\,d\mu_jd\mu_k
 \right|
 \lesssim_\alpha \frac{r_j^3r_k^3}{M^8}
 \left(\frac{r_k}{M}\right)^{\alpha-8}
 =\frac{r_j^3}{M^\alpha r_k^{5-\alpha}}.
\]

\noindent For the pairs in \eqref{eq:oppositeblock}, note that if $M/2<j\leq M-1$ then $r_j\geq 2M-4$ and since $r_M\leq 12M+3$, we have $r_M\leq 8r_j$ whenever $M\geq16$. Hence, the scales are comparable, a case that has already been covered in \eqref{eq:comparableblock}. We can thus assume by symmetry that $j\leq M/2$, which implies $H_j\geq1/2$. More precisely,
$s\ge H_j\ge3/4-1/(2M)\ge23/32>2/3$.
For opposite noncentral bands $t\le0$, while for the equatorial band
$t\le r_M/N\le(12M+3)/(4M^2)\le195/1024<1/5$.
By monotonicity,
\[
 \frac{s-t}{1-st}\ge\frac7{13},\qquad
 \frac VU=\sqrt{1-\left(\frac{s-t}{1-st}\right)^2}
 \le\frac{\sqrt{120}}{13}<\frac{15}{16}.
\]
We therefore have $\frac85\leq U\leq 4$
and $V\le15U/16$. Again, applying
\eqref{eq:two-variable-taylor} and \eqref{eq:separated-derivative} gives
\[
 \left|
 \int_{B_j}\int_{B_k}F_\alpha\,d\mu_jd\mu_k
 \right|
 \lesssim_\alpha\frac{r_j^3r_k^3}{M^8}.
\]
This proves the lemma.

\subsection*{AI usage disclosure}
We have used LLMs to refine the computations, revise the correctness of the results and assistance with the Lean formalization. The authors take responsibility of all the content. 

The Lean formalization is available at the GitHub repository:
\href{https://github.com/joaquimortega/Rieszenergy}{Rieszenergy}.
The repository contains the proof blueprint, the Lean sources, a standalone
proof file, and the verification records. All blueprint obligations are
proved, and the public statements have been checked by Comparator and
Lean's kernel.

\end{document}